\documentclass[twoside]{article}

\usepackage{amssymb,amsthm}
\usepackage{amsmath}
\usepackage{mathabx}
\usepackage{fancyhdr}
\usepackage{url}
\usepackage{hyperref}
\usepackage{enumitem}

\newtheorem{thm}{Theorem}[section]
\newtheorem{cor}[thm]{Corollary}
\newtheorem{lem}[thm]{Lemma}
\newtheorem{defi}[thm]{Definition}
\newtheorem{prop}[thm]{Proposition}
\theoremstyle{remark}
\newtheorem{comentario}{Remark}

\newcounter{example}[section]
\makeatletter
\renewcommand{\p@example}{\thesection.} 
\makeatother
\newenvironment{example}{\refstepcounter{example}\noindent\textit{Example \arabic{section}.\arabic{example}}.}{ }

\makeatletter
\newcommand{\labitem}[2]{%
\def\@itemlabel{#1}
\item
\def\@currentlabel{#1}\label{#2}}
\makeatother
\makeatletter
\def\namedlabel#1#2{\begingroup
    #2%
    \def\@currentlabel{#2}%
    \phantomsection\label{#1}\endgroup
}
\makeatother

\title{Periodic solutions of
Euler-Lagrange equations in an anisotropic Orlicz-Sobolev space setting  }
\author{
 Fernando D. Mazzone \thanks{SECyT-UNRC, FCEyN-UNLPam and CONICET}\\
Dpto. de Matem\'atica, Facultad de Ciencias Exactas, F\'{\i}sico-Qu\'{\i}micas y Naturales\\
Universidad Nacional de R\'{i}o Cuarto\\
(5800) R\'{\i}o Cuarto, C\'ordoba, Argentina,\\
\url{fmazzone@exa.unrc.edu.ar} \\
Sonia Acinas \thanks{SECyT-UNRC,  FCEyN-UNLPam and UNSL}\\
Dpto. de Matem\'atica, Facultad de Ciencias Exactas y Naturales\\
Universidad Nacional de La Pampa\\
(L6300CLB) Santa Rosa, La Pampa, Argentina\\
\url{sonia.acinas@gmail.com}\\[3mm]
}

\date{}

\newcommand{\orlnor}{\|_{L^{\Phi}}}
\newcommand{\linf}{\|_{L^{\infty}}}
\newcommand{\lphi}{L^{\Phi}}
\newcommand{\lpsi}{L^{\Phi^{\star}}}
\newcommand{\ephi}{E^{\Phi}}
\newcommand{\claseor}{C^{\Phi}}
\newcommand{\wphi}{W^{1}\lphi}
\newcommand{\sobnor}{\|_{W^{1}\lphi}}
\newcommand{\domi}{\mathcal{E}^{\Phi}}
\renewcommand{\b}[1]{\boldsymbol{#1}}
\newcommand{\rr}{\mathbb{R}}

\renewcommand{\leq}{\leqslant} 
\renewcommand{\geq}{\geqslant} 
\newcommand{\epsi}{E^{\Phi^{\star}}}
\newcommand{\Phie}{\Phi^{\star}}
\newcommand{\lip}{\mathop{\rm Lip}}

\DeclareSymbolFont{symbolsC}{U}{txsyc}{m}{n}
\DeclareMathSymbol{\strictif}{\mathrel}{symbolsC}{74}

\begin{document}

\maketitle
\begingroup
    \renewcommand{\thefootnote}{}
    \footnotetext{%
    }%
        \footnotetext{%
    }%
    \endgroup

\begin{abstract}
In this paper we consider the problem of finding periodic solutions of certain Euler-Lagrange equations which include, among others, equations involving the $p$-Laplace operator and, more generally, the $(p,q)$-Laplace operator. We employ the direct method of the calculus of variations in the framework of anisotropic Orlicz-Sobolev spaces. These spaces appear to be useful in formulating a unified theory of existence of solutions for such a problem.
\end{abstract}

\pagestyle{fancy} \headheight 35pt \fancyhead{} \fancyfoot{}

\fancyfoot[C]{\thepage} \fancyhead[CE]{\nouppercase{F.D. Mazzone   and S. Acinas}} \fancyhead[CO]{\nouppercase{\section}}

\fancyhead[CO]{\nouppercase{\leftmark}}


\section{Introduction}

Let $\Phi:\mathbb{R}^d\to [0,+\infty)$ be  a differentiable, convex function such that $\Phi(0)=0$, $\Phi(y)>0$ if $y\neq 0$, $\Phi(-y)=\Phi(y)$,
 and
\begin{equation}\label{eq:N-sub-inf}
\lim_{|y|\to\infty}\frac{\Phi(y)}{|y|}=+\infty,
\end{equation}
where $|\cdot|$ denotes the euclidean norm on $\rr^d$. From now on, we say that $\Phi$ is an $N_{\infty}$ function if $\Phi$ satisfies the previous properties.

For $T>0$, we assume that  $F:[0,T]\times\rr^d\to\rr^d$  ($F=F(t,x)$)
  is a differentiable function  with respect to $x$ for a.e. $t\in [0,T]$. Additionally, suppose that $F$ satisfies the following conditions:
\begin{enumerate}
\labitem{(C)}{item:condicion_c} $F$ and its gradient $\nabla_x F$, with respect to $x\in\rr^d$,  are  Carath\'eodory functions, i.e. they are measurable functions with respect to $t\in [0,T]$, for every  $x\in\rr^d$, and they are continuous functions with  respect to  $x\in\rr^d$ for a.e. $t \in [0,T]$.
 \labitem{(A)}{item:condicion_a}  For   a.e. $t\in [0,T]$, it holds that
\begin{equation}\label{eq:phi-lagrange}
|F(t,x)| + |\nabla_x F(t,x)|  \leq a(x)b(t),
\end{equation}
where  $a\in C\left(\rr^d,[0,+\infty)\right)$ and $0\leq b\in L^1([0,T],\rr)$.
\end{enumerate}

The goal of this paper is to obtain existence of solutions for  the following problem:

\begin{equation}\label{eq:ProbPhiLapla}
    \left\{%
\begin{array}{ll}
  \frac{d}{dt} \nabla \Phi(u'(t))= \nabla_{x}F(t,u(t)), \quad \hbox{for a.e. } t \in (0,T),\\
    u(0)-u(T)=u'(0)-u'(T)=0.
\end{array}%
\right. \tag{${P_{\Phi}}$}
\end{equation}

Our approach involves the direct method of the calculus of variations in the framework of \emph{anisotropic Orlicz-Sobolev spaces}. 
We suggest the article  \cite{Orliczvectorial2005} for definitions and main results on anisotropic Orlicz spaces. These spaces allow us to unify and extend previous results on existence of solutions for systems like \eqref{eq:ProbPhiLapla}.
 We will find solutions of \eqref{eq:ProbPhiLapla} by finding extreme points of the \emph{action integral} 
\begin{equation}\label{eq:integral_accion}
  I(u):=\int_{0}^T \Phi(u'(t))+ F(t,u(t))\ dt.\tag{$IA$}
\end{equation}

In what follows, we shall denote by $\mathcal{L}=\mathcal{L}_{\Phi,F}$ the function $\Phi(y)+F(t,x)$, and we will call it \emph{Lagrangian}. 

The classic book  \cite{mawhin2010critical} deals mainly with problem \eqref{eq:ProbPhiLapla} with $\Phi(x)=\Phi_2(x):=|x|^2/2$, through various methods: direct, dual, saddle points,  minimax, etc. The results in \cite{mawhin2010critical} were extended and improved in several articles,  see  \cite{tang1998periodic,tang2001periodic,tang1995periodic,wu1999periodic,zhao2004periodic}  to cite some examples. The case $\Phi(x)=\Phi_p(y):=|y|^p/p$, for arbitrary $1<p<\infty$ were considered in  \cite{tang2010periodic,Tian2007192}, among other papers. In this case, \eqref{eq:ProbPhiLapla} is reduced to the \emph{$p$-laplacian system}.  If $\Phi_{p_1,p_2}:\rr^d\times \rr^d\to [0,+\infty)$  is defined by
\begin{equation}\label{eq:phip_1p_2}\Phi_{p_1,p_2}(y_1,y_2):=\frac{|y_1|^{p_1}}{p_1}+\frac{|y_2|^{p_2}}{p_2},\end{equation}
then \eqref{eq:ProbPhiLapla} becomes  $(p_1,p_2)$-laplacian system, see
\cite{li2014periodic,pasca2010periodic,pacsca2010some,pasca2011some,pasca2016periodic,yang2012periodic,yang2013existence}.  In a previous paper (see \cite{ABGMS2015}), we obtained similar results in an isotropic Orlicz framework. Hence \eqref{eq:ProbPhiLapla}  contains several problems that have been considered by many authors in the past. Our results still improve some results on $(p_1,p_2)$-laplacian systems since we obtain existence  of solutions for them under  less restrictive conditions.
For all this, we believe that anisotropic Sobolev-Orlicz spaces can provide a suitable framework to  unify many known results. On the other hand, we point out that one of the most important aspects in our work  is the possibility of dealing with  functions $\Phi$ that grow faster than power functions. 

\begin{example}\label{ex:RandEx} As an illustrative example, we obtain existence of solutions for
\begin{equation}\label{eq:RandExample}
    \left\{%
\begin{array}{ll}
  \frac{d}{dt} \left[u_1(t)e^{(u_1'(t))^2+(u_2'(t))^2}\right]= F_{x_1}(t,u(t)), \quad \hbox{for a.e. } t \in (0,T),\\
   \frac{d}{dt} \left[u_2(t)e^{(u_1'(t))^2+(u_2'(t))^2}\right]= F_{x_2}(t,u(t)), \quad \hbox{for a.e. } t \in (0,T),\\
    u(0)-u(T)=u'(0)-u'(T)=0,
\end{array}%
\right. 
\end{equation}
where  $F(t,x)=P(t)Q(x)$, with $P$ and $Q$ polynomials  (see Remark \ref{com:ejemplo} below).
\end{example}

The paper is organized as follows. In Section \ref{sec:preliminares},  we summarize some known results about Orlicz and Orlicz-Sobolev spaces. In order to obtain existence of minimizers of action integrals it is necessary that the functional $I$ be coercive. 
In the past, several conditions on $F$ have been useful to obtain coercivity of $I$ for the functions  $\Phi_p$ and $\Phi_{p_1,p_2}$ . 
In this paper  we investigate the condition that in the literature was called sublinearity 
(see \cite{tang1998periodic,wu1999periodic,zhao2004periodic} for the laplacian, \cite{li2015infinitely,tang2010periodic} for the $p$-laplacian and \cite{li2014periodic,pasca2010periodic,pacsca2010some,yang2013existence} for $(p_1,p_2)$-laplacian). In Section \ref{sec:existencia},  we contextualize the sublinearity within our framework (see \eqref{eq:cond-sub} below) and we  establish results of existence of minimizers of \eqref{eq:integral_accion}  in Theorem \ref{coercitividad-r}.   In  Section \ref{sec:diferenciabilidad},  we establish conditions under which a minimum of \eqref{eq:integral_accion}  is a solution of \eqref{eq:ProbPhiLapla}.

\section{Anisotropic Orlicz and Orlicz-Sobolev spaces}\label{sec:preliminares}

In this section, we give a short introduction to  Orlicz and Orlicz-Sobolev spaces of vector valued functions associated to anisotropic $N_{\infty}$ functions $\Phi:\rr^d\to[0,+\infty)$.  References for  these topics are \cite{chamra2017anisotropic,cianchi2000fully,cianchi2004optimal,Desch2001,gwiazda2013anisotropic,Orliczvectorial2005,Skaff1969,trudinger1974imbedding}. For the theory of convex functions in general we suggest \cite{clarke2013functional}. 
Note that, unlike in \cite{gwiazda2013anisotropic}, we do not require that $N_{\infty}$ functions be superlinear near  0, i.e. $\Phi(x)/|x|\to 0$ when $|x|\to 0$. However, most of the results proved in \cite{gwiazda2013anisotropic} do not depend on this property. 

If $\Phi(y)$ is an $N_{\infty}$  function which depends on $|y|$ ($\Phi(y)=\overline{\Phi}(|y|)$), then we say that $\Phi$ is \emph{radial}. 

We can use the following example to obtain new $N_{\infty}$ functions from given $N_{\infty}$ ones.

    \begin{example}
      Let  $(d_1,\ldots,d_k)\in\mathbb{Z}^k_+$. Suppose that  $\Phi_j:\rr^{d_j}\to [0,+\infty)$, $j=1,\ldots k$,  are $N_{\infty}$ functions and  $O_j\in L(\rr^{d},\rr^{d_j})$ are bounded linear functions  satisfying $\bigcap_{j=1}^k\ker O_j=\{0\}$. Then 
      
      \[\Phi(y):=\sum_{j=1}^k\Phi_j(O_jy)\]
      is an $N_{\infty}$ function. 
      
      Let us briefly show that $\Phi$ satisfies \eqref{eq:N-sub-inf}. Suppose that $|y_n|\to\infty$ and $\Phi(y_n)/|y_n|$ is bounded. If for some $j=1,\ldots,k$ there exist  $\epsilon>0$ and a subsequence $n_s$ such that $|O_jy_{n_s}|\geq \epsilon |y_{n_s}|$, then $\Phi_j(O_jy_{n_s})/|y_{n_s}|\to\infty$, contrary to our assumption. Hence $O_jy_n/|y_n|\to 0$ when $n\to\infty$. Passing to a subsequence, we can assume that there exists $y\in\rr^d$ such that  $y_n/|y_n|\to y$. Then $y\in\ker O_j$ and $y\neq 0$, which is a contradiction.
      
     As a consequence, the function $\Phi:\rr^d\times\rr^d\to [0,+\infty)$ defined by
      \[\Phi(y_1,y_2)=e^{|y_1-y_2|}-1+|y_2|^p,\]
      with $1<p<\infty$ is an $N_{\infty}$ function.

    \end{example}

Associated to $\Phi$ we have the \emph{complementary function} $\Phi^{\star}$ which is defined at $\zeta\in\rr^d$ as
\begin{equation}\label{eq:conjugada}
 \Phi^{\star}(\zeta)=\sup\limits_{y\in\mathbb{R}^d} y\cdot \zeta-\Phi(y).
\end{equation}
From the continuity of $\Phi$ and  \eqref{eq:N-sub-inf}, we also have that $\Phi^{\star}:\rr^d \to [0,\infty)$. 
The complementary function $\Phi^{\star}$ is an  $N_{\infty}$ function (see \cite[Ch. 2]{mawhin2010critical} and \cite[Thm. 2.2]{Orliczvectorial2005}).  Now, Moreau Theorem (see \cite[Thm. 4.21]{clarke2013functional}) implies that $\Phi^{\star\star}=\Phi$.

 Some useful properties which are satisfied by $N_{\infty}$ functions are:

\begin{enumerate}
 \item[\namedlabel{eq:escalar_ine}{(P1)}]   $\Phi(\lambda x)\leq \lambda\Phi(x)$, for every  $\lambda\in[0,1],x\in\rr^d$;
 \item[\namedlabel{eq:escalar_ine_2}{(P2)}] if $ 0<|\lambda_1|\leq |\lambda_2| $, then $\Phi(\lambda_1 x)\leq\Phi(\lambda_2 x);$
  \item[\namedlabel{eq:young-ine}{(P3)}] $x\cdot y\leq \Phi(x)+\Phi^{\star}(y)$;
  \item[\namedlabel{eq:young-id}{(P4)}] $x\cdot \nabla\Phi(x)= \Phi(x)+\Phi^{\star}(\nabla \Phi (x))$.

\end{enumerate}

We say that  $\Phi:\mathbb{R}^d\rightarrow [0,+\infty)$ satisfies the  \emph{$\Delta_2$-condition} and we denote  $\Phi \in \Delta_2$,
if there exists a  constant $C>0$  such that
\begin{equation}\label{delta2defi}\Phi(2x)\leq C \Phi(x)+1,\quad x\in\rr^d.
\end{equation}

Throughout this article, we denote by $C=C(\lambda_1,\ldots,\lambda_n)$ a positive constant that may depend on $T$, $\Phi$ (or another $N_{\infty}$ functions) and the parameters $\lambda_1,\ldots,\lambda_n$ . We assume that the value that $C$ represents may change in 
different occurrences in the same chain of inequalities.

If $\Phi$ satisfies the $\Delta_2$-condition, then $\Phi$ satisfies the following properties:
\begin{enumerate}
 \item[\namedlabel{eq:quasi-sub-aditividad}{(P5)}] There exists $C>0$ such that for every $x,y\in \rr^d$, $\Phi(x+y)\leq C (\Phi(x)+\Phi(y))+1.$
\item[\namedlabel{{eq:delta2-lambda}}{(P6)}]  For any $\lambda>1$ there exists $C(\lambda)>0$ such that $\Phi(\lambda x)\leq C(\lambda)\Phi(x)+1.$
\item [\namedlabel{{eq:delta2-power}}{(P7)}] There exist $1<p<\infty$ and $C>0$ such that $\Phi(x)\leq C|x|^p+1$.

\end{enumerate}

Let $\Phi_1$ and $\Phi_2$ be   $N_{\infty}$ functions. Following  \cite{trudinger1974imbedding}, we write $\Phi_1\strictif\Phi_2$ if there exist  $k,C>0$ such that
\begin{equation}\label{eq:orden} \Phi_1(x)\leq C+\Phi_2(kx),\quad x\in\rr^d.\end{equation}
For example, if $\Phi \in \Delta_2$ then there exists $p\in (1,+\infty)$ such that $\Phi\strictif |x|^p$.  If for every $k>0$ there exists $C=C(k)>0$ such that \eqref{eq:orden} holds, we write  $\Phi_1\llcurly\Phi_2$.  We observe that  $\Phi_1\strictif \Phi_2$ implies that $\Phi^{\star}_2\strictif\Phi^{\star}_1$. A similar assertion holds for relation $\llcurly$.

If $\Phi^{\star}\in\Delta_2$ then $\Phi$ satisfies the \emph{$\nabla_2$-condition}, i.e.  for every $0<r<1$ there exist $l=l(r)>0$ and $C'=C'(r)>0$ such that 
\begin{equation}\label{eq:nabla2}
  \Phi(x)\leq \frac{r}{l}\Phi(l x)+C',\quad x\in\rr^d.
\end{equation}

 We denote by $\mathcal{M}:=\mathcal{M}\left([0,T],\rr^d\right)$, with $d\geq 1$,  the set of all measurable functions (i.e. functions which are limits of simple functions)  defined on $[0,T]$ with values on $\mathbb{R}^d$ and  we write $u=(u_1,\dots,u_d)$ for  $u\in \mathcal{M}$.

 Given  an $N_{\infty}$ function $\Phi$ we define the \emph{modular function} 
$\rho_{\Phi}:\mathcal{M}\to \mathbb{R}^+\cup\{+\infty\}$ by
\[\rho_{\Phi}(u):= \int_0^T \Phi(u)\ dt.\]

Now, we introduce the \emph{Orlicz class} $C^{\Phi}=C^{\Phi}\left([0,T],\rr^d\right)$   by setting
\begin{equation}\label{claseOrlicz}
  C^{\Phi}:=\left\{u\in \mathcal{M} | \rho_{\Phi}(u)< \infty \right\}.
\end{equation}
The \emph{Orlicz space} $\lphi=L^{\Phi}\left([0,T],\rr^d\right)$ is the linear hull of $\claseor$;
equivalently,
\begin{equation}\label{espacioOrlicz}
\lphi:=\left\{ u\in \mathcal{M}| \exists \lambda>0: \rho_{\Phi}(\lambda u) < \infty   \right\}.
\end{equation}
  The Orlicz space $\lphi$ equipped with the \emph{Luxemburg norm}
\[
\|  u  \orlnor:=\inf \left\{ \lambda\bigg| \rho_{\Phi}\left(\frac{v}{\lambda}\right) dt\leq 1\right\},
\]
is a Banach space.

The subspace $\ephi=\ephi\left([0,T],\rr^d\right)$ is defined as the closure in $\lphi$ of the subspace $L^{\infty}\left([0,T],\rr^d\right)$ of all $\mathbb{R}^d$-valued essentially bounded functions. The equality $\lphi=\ephi$ is true if and only if $\Phi\in\Delta_2$ (see \cite[Cor. 5.1]{Orliczvectorial2005}). 

A generalized version of \emph{H\"older's inequality} holds in Orlicz spaces (see \cite[Thm. 7.2]{Orliczvectorial2005}). Namely, if $u\in\lphi$ and $v\in\lpsi$ then $u\cdot v\in L^1$ and
\begin{equation}\label{holder}
\int_0^Tv\cdot u\ dt\leq 2 \|u\orlnor\|v\|_{L^{\Phi^{\star}}}.
\end{equation}
By $u\cdot v$ we denote the usual dot product in $\mathbb{R}^{d}$ between $u$ and $v$.

We consider the subset $\Pi(\ephi,r)$ of $\lphi$ given by
\[\Pi(\ephi,r):=\{u\in\lphi| d(u,\ephi)<r\}=\{u\in\lphi| d(u,L^{\infty})<r\}.\]
This set is related to the Orlicz class $\claseor$ by the following inclusions
\begin{equation}\label{eq:inclusiones}\Pi(\ephi, r )\subset r \claseor\subset\overline{\Pi(\ephi,r)}
\end{equation}
for any positive $r$. This relation is a trivial generalization of  \cite[Thm. 5.6]{Orliczvectorial2005}.
If $\Phi \in \Delta_2$,  then the sets $\lphi$, $\ephi$, $\Pi(\ephi,r)$ and $\claseor$ are equal.
 
As usual, if $(X,\|\cdot\|_X)$ is a normed space and $(Y,\|\cdot \|_Y)$ is a linear subspace of $X$,  we write $Y\hookrightarrow X$ and we say that $Y$ is \emph{embedded} in $X$  when there exists $C>0$ such that
$\|y\|_X\leq C\|y\|_Y$ for any $y\in Y$.  With this notation, H\"older's inequality states that  $\lphi\hookrightarrow  \left[\lpsi\right]^\star$, where a function $v\in\lphi$ is associated  to $\xi_v\in \left[\lpsi\right]^\star$ given by
\begin{equation}\label{pairing}
  \langle \xi_v,u\rangle=\int_0^Tv\cdot u\ dt.
\end{equation}

It is easy to prove that $L^{\infty} \hookrightarrow L^\Phi\hookrightarrow L^1$ for any $N_{\infty}$ function $\Phi$.

Suppose $u\in\lphi([0,T],\rr^d)$ and consider $K:=\rho_{\Phi}(u)+1\geq 1$. Then, from \ref{eq:escalar_ine} we have $\rho_{\Phi}(K^{-1}u)\leq K^{-1}\rho_{\Phi}(u)\leq 1$. Therefore, we conclude 
\begin{equation}\label{eq:amemiya}
 \|u\orlnor \leq \rho_{\Phi}(u)+1.
\end{equation}


 

We highlight that $\lphi\left([0,T],\rr^d\right)$ can be equipped with the weak$\star$ topology induced by $\epsi\left([0,T],\rr^d\right)$ because 
$\lphi\left([0,T],\rr^d\right)=\left[\epsi\left([0,T],\rr^d\right)\right]^{\star}$ (see \cite[Thm. 3.3]{gwiazda2013anisotropic}).

We define the \emph{Sobolev-Orlicz space} $\wphi\left([0,T],\rr^d\right)$ by
\[\wphi\left([0,T],\rr^d\right):=\left\{u| u\in AC\left([0,T],\rr^d\right) \hbox{ and } u'\in \lphi\left([0,T],\rr^d\right)\right\},\]
where $AC\left([0,T],\rr^d\right)$ denotes the space of all $\rr^d$ valued absolutely continuous functions defined on $[0,T]$. The space $\wphi\left([0,T],\rr^d\right)$ is a Banach space when equipped with the norm
\begin{equation}\label{def-norma-orlicz-sob}
\|  u  \|_{\wphi}= \|  u  \|_{\lphi} + \|u'\orlnor.
\end{equation}

Let the function $A_{\Phi}:\rr^d\to [0,+\infty)$ be the greatest convex radial minorant of $\Phi$, i.e.
\begin{equation}\label{eq:inversa-gral}
A_{\Phi}(x)=\sup\left\{\Psi(x) \right\},
\end{equation}
where the supremum is taken over all the convex, non negative, radial functions $\Psi$ with $\Psi(x)\leq \Phi(x)$.

\begin{prop}\label{prop:AsubPhi}  $A_{\Phi}$ is a radial and $N_{\infty}$ function.
\end{prop}

\begin{proof} The convexity and radiality of $A_{\Phi}$ is a consequence of the fact that the supremum preserves these properties. Then, it is only necessary to show that $A_{\Phi}(x)>0$ when $x\neq 0$, and  $A_{\Phi}(x)/|x|\to\infty$, when $|x|\to\infty$. We write, for $r\in\rr$, $r^+=\max\{r,0\}$. Since $\Phi$ is an $N_{\infty}$ function,  for every $k>0$ there exists $r_0>0$ such that  $\Phi(x)\geq k(|x|-r_0)^+$, for $|x|>r_0$.  As $ k(|x|-r_0)^+$ is a non negative, radial, convex function, it follows that $A_{\Phi}(x)\geq k(|x|-r_0)^+$. Therefore $\liminf_{|x|\to\infty} A_{\Phi}(x)/|x|\geq k$ and consequently   $\lim_{|x|\to\infty} A_{\Phi}(x)/|x|=\infty$.

As $\Phi$ is an
$N_{\infty}$ continuous function, for every $r>0$ there exists $k(r)>0$ such that $\Phi(x)\geq k(r)|x|\geq k(r)(|x|-r)^+$, when $|x|\geq r$. This fact implies that $A_{\Phi}(x)>0$ for $x\neq 0$.
\end{proof}
By abuse of notation, we identify $A_{\Phi}$ with a function defined on $[0,+\infty)$.  This function is invertible.
\begin{cor}\label{cor:incr_aphi} $\lphi([0,T],\rr^d) \hookrightarrow L^{A_{\Phi}}([0,T],\rr^d)$.

\end{cor}

 As is customary, we will use the decomposition $u=\overline{u}+\widetilde{u}$ for a function $u\in L^1([0,T])$  where $\overline{u} =\frac1T\int_0^T u(t)\ dt$ and $\widetilde{u}=u-\overline{u}$.

\begin{lem}\label{lem:inclusion orlicz} Let $\Phi:\rr^d\to [0,+\infty)$ be an $N_{\infty}$
function and let \linebreak[4]$u\in\wphi\left([0,T],\rr^d\right)$. Let 
$A_{\Phi}: \rr^d \to  [0,+\infty)$ be the isotropic function defined by \eqref{eq:inversa-gral}. Then:
 
\begin{enumerate}
  \item \emph{Morrey's inequality}.   For every $s,t\in [0,T]$, $s\neq t$
  \begin{equation}
   |u(t)-u(s)| \leq
  |s-t|A_{\Phi}^{-1}\left(\frac{1}{|s-t|}\right)\|u'\orlnor.\tag{M.I}\label{in-sob-cont}
  \end{equation}

  \item \emph{Sobolev's inequality}. 
  \begin{equation}
   \|u\linf \leq A_\Phi^{-1}\left(\frac{1}{T}\right)\max\{1,T\}\|u\sobnor.\tag{S.I}\label{eq:sobolev}
  \end{equation}

  \item \emph{Poincar\'e-Wirtinger's inequality}. We have $\widetilde{u}\in L^{\infty}\left([0,T],\rr^d\right)$ and 
    \begin{equation}\label{eq:wirtinger-iso}
    \|\widetilde{u}\|_{L^{\infty}} \leq T A_{\Phi}^{-1}\left(\frac{1}{T}\right)\| u'\orlnor.\tag{P-W.I}
    \end{equation}
    
    \item\label{it:embeding} If $\Phi$ is an $N_{\infty}$ function, then the space $\wphi\left([0,T],\rr^d\right)$ is compactly embedded in the space of  continuous functions $C([0,T],\rr^d)$.
  \end{enumerate}

\end{lem}

\begin{proof}  It is an immediate consequence of Corollary \ref{cor:incr_aphi} and \cite[Lemma 2.1, Cor. 2.2]{ABGMS2015}.
\end{proof}

Lemma \ref{lem:inclusion orlicz} gives us estimates for isotropic norms of $u$. In these type of inequalities some information is lost.  The following result gives us an estimate that takes into account the anisotropic nature of the space $\wphi\left([0,T],\rr^d\right)$. The proof is similar to  that of  \cite[Thm. 4.5]{chamra2017anisotropic}.

\begin{lem}[Anisotropic Poincar\'e-Wirtinger's inequality]\label{lem:inclusion orliczII} Let $\Phi:\rr^d\to [0,+\infty)$ be an $N_{\infty}$
function and let $u\in\wphi\left([0,T],\rr^d\right)$. Then

\begin{equation}\label{eq:wirtinger}
  \Phi\left(\tilde{u}(t)\right)\leq\frac{1}{T} \int_0^T \Phi\left(Tu'(r)\right)dr.\tag{A.P-W.I}
\end{equation}

\end{lem}

\begin{proof}  Applying Jensen's inequality twice, we get
\begin{equation*}
\begin{split}
\Phi(\tilde{u}(t))&=\Phi\left(\frac{1}{T}\int_0^T \left(u(t)-u(s)\right) ds\right)\\
&\leq\frac{1}{T}\int_0^T \Phi(u(t)-u(s))ds\\
&\leq \frac{1}{T}\int_0^T \Phi\left(\int_s^t |t-s| u'(r)\frac{dr}{|t-s|}\right)ds\\
&\leq 
\frac{1}{T}\int_0^T \frac{1}{|t-s|} \int_s^t\Phi\left(|t-s| u'(r)\right)drds.
\end{split}
\end{equation*}
From \ref{eq:escalar_ine} we have that $\Phi(rx)/r$ is increasing with respecto to $r>0$ for a fixed $x\in\rr^d$. Therefore, previous inequality  implies \eqref{eq:wirtinger}.
\end{proof}

\begin{comentario}\label{com:equiv-norm} As a consequence of Lemma \ref{lem:inclusion orlicz}, we obtain that 
  \[\|u\sobnor'=|\overline{u}|+\|u'\orlnor,\]
  defines an equivalent norm to $\|\cdot\sobnor$ on $\wphi([0,T],\rr^d)$.   
\end{comentario}
Another immediate consequence of  Lemma \ref{lem:inclusion orlicz} is the following result.
 \begin{cor}\label{cor:unif_conv} Every bounded sequence $\{u_n\}$ in  $\wphi([0,T],\rr^d)$  has an uniformly convergent subsequence. 
\end{cor}

\section{Existence of minimizers}\label{sec:existencia}

It is well known that an important ingredient in the direct method of the calculus of variations is the coercivity of action integrals. In order to obtain  coercivity for the action integral $I$, defined in  \eqref{eq:integral_accion}, it is necessary to impose more restrictions on the potential $F$.

There are several restrictions that were explored in the past. The one we will study in this article is based on what is known in the literature as sublinearity (see \cite{tang1998periodic,wu1999periodic,zhao2004periodic} for the Laplacian, \cite{tang2010periodic,li2015infinitely} for the $p$-Laplacian and \cite{yang2013existence,li2014periodic,pasca2010periodic,pacsca2010some} for $(p_1,p_2)$-Laplacian). In the current article we will use another denomination for this property.

\begin{defi} Let $F:[0,T]\times \rr^d\to\rr$ be a function satisfying \ref{item:condicion_c} and \ref{item:condicion_a}. We say that $F$ satisfies condition \eqref{eq:cond-sub} if there exist an $N_{\infty}$ function $\Phi_0$, with $\Phi_0 \llcurly \Phi$ and
a function $d \in  L^1([0,T],\rr)$, with $d\geq 1$, such that

\begin{equation}\label{eq:cond-sub}
  \Phi^{\star}\left(\frac{\nabla_x F}{d(t)}\right)\leq \Phi_0(x)+1.\tag{$B$}
\end{equation}
\end{defi}

The condition \eqref{eq:cond-sub} encompasses the sublinearity condition as it was introduced in the context of $p$-laplacian or $(p_1,p_2)$-laplacian systems. For example, in
 \cite[Thm. 1.1.]{li2014periodic} Li, Ou and Tang considered a potential $F:[0,T]\times\rr^d\times\rr^d\to\rr$ satisfying \ref{item:condicion_c} and \ref {item:condicion_a} and the following condition  (we recall that $p'=p/(p-1)$).
\begin{enumerate}   
  \labitem{(H)}{item:condicion_h1} There exist $f_i,g_i,h_i\in L^1([0,T],\rr_+)$,  $\alpha_i\in [0,p_i/p_i')$, $i=1,2$,  $\beta_1\in [0,p_2/p_1')$ and $\beta_2\in [0,p_1/p_2')$ such that
 \begin{equation*}
    \begin{array}{cc}
    |\nabla_{x_1}F(t,x_1,x_2)|&\leq f_1(t)|x_1|^{\alpha_1}+g_1(t)|x_2|^{\beta_1}+h_1(t),\\
    |\nabla_{x_2}F(t,x_1,x_2)|&\leq f_2(t)|x_2|^{\alpha_2}+g_2(t)|x_1|^{\beta_2}+h_2(t).
    \end{array}
 \end{equation*}
\end{enumerate}
We leave to the reader to prove that \ref{item:condicion_h1} implies \eqref{eq:cond-sub}, with $\Phi=\Phi_{p_1,p_2}$,
$\Phi_0=\Phi_{\overline{p}_1,\overline{p}_2}$, where $\overline{p}_i$, $i=1,2$, are taken so that  $\max\{\alpha_1p_1',\beta_2p_2'\}\leq \overline{p}_1<p_1$ and $\max\{\alpha_2p_2',\beta_1p_1'\}\leq \overline{p}_2<p_2$ and  $d=C(1+\sum_i \{f_i+g_i+h_i\})\in L^1$, with  $C>0$ chosen large enough.

\begin{thm}\label{coercitividad-r}
  Let $\Phi$ be an $N_{\infty}$-function whose complementary function $\Phi^{\star}$ satisfies the $\Delta_2$-condition. Let $F$ be a potential that satisfies \ref{item:condicion_c}, \ref{item:condicion_a},\eqref{eq:cond-sub} and the following condition
\begin{equation}\label{eq:propiedad-coercividad-phi0}
\lim_{|x|\to\infty}\frac{\int_{0}^{T}F(t,x)\ dt}{\Phi_0(2x)}=+\infty.
\end{equation}
Let $M$ be a weak${\star}$ closed subspace of $\lphi$ and let $V\subset C([0,T],\rr^d)$ be closed in the $C([0,T],\rr^d)$-strong topology. Then  $I$ attains a minimum on $H=\{u\in \wphi | u\in V\text{ and } u'\in M\}$.
\end{thm}

\begin{proof} \emph{Step 1. The action integral is coercive}.

  Let $\lambda$ be any positive number with $\lambda>2\max\{T,1\}$. Since $\Phi_0\llcurly \Phi$ there exists $C(\lambda)>0$ such that 
  \begin{equation}\label{eq:cond_C}
    \Phi_0(x)\leq \Phi\left(\frac{x}{2\lambda}\right)+C(\lambda),\quad  x\in\rr^d. 
  \end{equation}
By the decomposition $u=\overline{u}+\tilde{u}$, the absolutely continuity of $F(t,x+sy)$ with respect to $s\in\rr$,  Young's inequality, \eqref{eq:cond-sub}, the convexity of $\Phi_0$, \ref{eq:escalar_ine_2}, \eqref{eq:cond_C} and \eqref{eq:wirtinger}  we obtain
\begin{equation*}\label{eq:cota-dif-F}
\begin{split}
J&:=\left|\int_0^T F(t,u)-F(t,\overline{u})dt\right|
\\
&\leq \int_0^T \int_0^1 |\nabla_x F(t,\overline{u}+s\tilde{u})\tilde{u}|dsdt
\\
&\leq
\lambda \int_0^T d(t) \int_0^1 \left[ \Phi^{\star}\left(d^{-1}(t)\nabla_xF(t,\overline{u}+s\tilde{u})\right)+\Phi\left(\frac{\tilde{u}}{\lambda}\right)\right]dsdt
\\
&\leq
\lambda\int_0^T d(t)\int_0^1\left[ \frac12\Phi_0(2\overline{u})+\frac12\Phi_0(2\tilde{u})ds+\Phi\left(\frac{\tilde{u}}{\lambda}\right)+1 \right]dsdt
\\
&\leq
\lambda\int_0^T d(t)\int_0^1\left[ \Phi_0(2\overline{u})+2\Phi\left(\frac{\tilde{u}}{\lambda}\right)+C(\lambda) \right]dsdt
\\
&\leq C_1 \Phi_0(2\overline{u})+\lambda C_2 \int_0^T \Phi\left(\frac{Tu'(s)}{\lambda}\right)ds+C_1,
\end{split}
\end{equation*}
where $C_2=C_2(\|d\|_{L^1})$ and $C_1=C_1(\|d\|_{L^1},\lambda)$.  Since $\Phi^{\star} \in \Delta_2$ we can choose $\lambda$ large enough so that $l=\lambda T^{-1}$ satisfies  \eqref{eq:nabla2} for $r=\frac12 \min\{(C_2T)^{-1},1\}$. Thus, we have
\[
 J\leq C_1 \Phi_0(2\overline{u})+\frac12 \int_0^T \Phi\left(u'(s)\right)ds+C_1.
\]
Then
\begin{equation}
\begin{split}
I(u)&=\int_0^T  \Phi(u')+F(t,u)dt\\
&= \int_0^T \{\Phi(u')+[F(t,u)-F(t,\overline{u})]+F(t,\overline{u})\}dt
\\
&\geq
\frac12\int_0^T \Phi(u')dt-C_1\Phi_0(2 \overline{u})+\int_0^T F(t,\overline{u})dt-C_1
\end{split}
\end{equation}

We take $u_n\in\wphi$ with $\|u_n\sobnor\to\infty$. From Remark \ref{com:equiv-norm}, we can suppose that $\|u_n'\orlnor \to \infty$ or $|\overline{u}_n|\to \infty$ . In the first case,from \eqref{eq:amemiya}  we have  that $\rho_\Phi(u_n)\to\infty$ and hence  $I(u_n)\to\infty$. In the second case,  $I(u_n)\to\infty$ as a consequence of \eqref{eq:propiedad-coercividad-phi0}.

\emph{Step 2. Suppose that $u_n\to u$ uniformly and $u_n'\overset{\star}{\rightharpoonup} u'$ in $\lphi([0,T],\rr^d)$ then
$
I(u)\leq\liminf_{n\to\infty}I(u_n)
$.
}

  Without loss of generality, passing to subsequences, we may assume that the $\liminf$ is really a $\lim$. The embedding  $\lphi([0,T],\rr^d)\hookrightarrow L^1([0,T],\rr^d)$ implies that  $u_n'\rightharpoonup u' $ in $ L^1([0,T],\rr^d)$.  Now, applying \cite[Thm. 3.6]{buttazzo1998one} we obtain $I(u)\leq \lim_{n\to\infty}I(u_n)$.

\emph{ Final step.}
The proof of the theorem is concluded with a usual argument. We take a minimizing sequence $u_n \in H$ of $I$. 
From the coercivity of $I$ we have that $u_n$ is bounded on $\wphi([0,T],\rr^d)$. By Corollary \ref{cor:unif_conv}  (passing to subsequences) we can suppose that $u_n$ converges uniformly to a function $u\in V$. 
On the other hand, $u'_n$ is bounded on $\lphi=\left[\epsi\right]^\star$. Thus,  
since $E^{\Phi^{\star}}$ is separable (see \cite[Thm. 6.3]{Orliczvectorial2005}), from \cite[Cor. 3.30]{brezis2010functional} 
it follows that there exist a subsequence of $u'_n$ (we denote it $u'_n$ again) and $v \in M$ such that $u'_n\overset{\star}{\rightharpoonup}v$. From this fact and the uniform convergence of $u_n$ to $u$, we obtain that
\[
\int_0^T\varphi'\cdot u\ dt=\lim_{n\to\infty}\int_0^T\varphi'\cdot u_n \ dt=
-\lim_{n\to\infty}\int_0^T\varphi\cdot u'_n\ dt=-\int_0^T\varphi\cdot v\ dt,
\]
for every function $\varphi\in C^{\infty}([0,T],\rr^d)\subset\epsi$ with $\varphi(0)= \varphi(T)=0$.
Thus, $u$ has a  derivative in the weak sense in $\lphi$. Taking account of 
$\lphi \hookrightarrow L^1$ and \cite[Thms. 2.3 and 2.17 ]{buttazzo1998one}, 
we obtain $u\in \wphi$ and $v=u'$ a.e. $t\in [0,T]$.
Hence, $u \in H$. 

Finally, the semicontinuity of $I$ that was established in step 2  implies that $u$ is a minimum of $I$.
\end{proof}

\begin{comentario} The results of this section can be extended without difficulty to any Lagrangian $\mathcal{L}$ with $\mathcal{L}\geq \mathcal{L}_{\Phi,F}$ (see \cite{ABGMS2015}).
\end{comentario}

\section{Regularity of minimizers and solutions of Euler-Lagrange equations}\label{sec:diferenciabilidad}

In this section, we will address the question of when minimizers of $I$ are solutions of the associated  Euler-Lagrange equations. It is a classic result that minimizers satisfying an a priori smothness condition (e.g. Lipschitz continuity) are solutions of them.  In Theorem \ref{th:tor_prin} we obtain a better a priori condition for the action integral under consideration.

We denote by $\lip([0,T],\rr^d)$ the set of $\rr^d$-valued Lipschitz continuous functions defined on $[0,T]$. If $X\subset \lphi([0,T],\rr^d)$ and $u\in \lphi([0,T],\rr^d)$,  we denote by $d(u,X)$ the distance from $u$ to $X$ computed with respect  to the Luxemburg norm. We recall that  $u\in\lip([0,T],\rr^d)$ implies that $d(u',L^{\infty}([0,T],\rr^d))=0$. The following is the main result of this section.

\begin{thm}\label{th:tor_prin} Assume   that $F$ is as in Theorem \ref{coercitividad-r} and $\Phi$ is  strictly convex. If $u$ is a minimum of $I$ on the set $H=\{u\in\wphi([0,T],\rr^d)| u(0)=u(T)\}$ and $d(u',L^{\infty}([0,T],\rr^d))<1$, then $u$ is solution of \eqref{eq:ProbPhiLapla}. 
\end{thm}

\begin{comentario} We observe that $H=\{u\in \wphi | u\in V\text{ and } u'\in M\}$ where
\[V:=\{u\in C([0,T],\rr^d)|u(0)=u(T)\},\quad M:=\lphi([0,T],\rr^d),\]
 and $V$ is $C([0,T],\rr^d)$-closed. Therefore, from the results of previous section  the functional $I$ given by \eqref{eq:integral_accion}, has a minimum $u$ on  $H$ and  $d(u',L^{\infty})\leq 1$. The last inequality   follows from  $\rho_{\Phi}(u')<\infty$ and \eqref{eq:inclusiones}. Then, the possible minima of $I$ that do not satisfy the hypotheses of  Theorem \ref{th:tor_prin} lie on the nowhere dense set $ \{u:d(u',L^{\infty})=1\}$ .

\end{comentario}

The proof of the previous theorem depends on the G\^ateaux differentiability of the action integral on the space $\wphi([0,T],\rr^d)$. We will deal with a more general lagrangian function $\mathcal{L}:[0,T]\times\rr^d\times\rr^d\to\rr$, which  is assumed  measurable in $t$ for each $(x,y)\in \rr^d\times\rr^d$ and  continuously differentiable at $(x,y)$ for almost every $t \in [0,T]$. We consider 
\begin{equation}\label{eq:integral_accion_gen} I(u)=I_{\mathcal{L}}(u)=\int_0^T\mathcal{L}(t,u(t),u'(t))dt,\tag{$IG$}
\end{equation}
 the action integral associated to $\mathcal{L}$. In order to obtain differentiability of $I$, it is necessary to impose some constraints on $\mathcal{L}$. In the paper \cite{chamra2017anisotropic}, Chmara and Maksymiuk  obtained differentiability for $I$ on $\wphi$ assuming  a similar condition  to Definition \ref{def:estructura} and additionally $\Phi\in\Delta_2\cap\nabla_2$. For our purpose, the condition $\Phi\in\Delta_2$ is a very serious limitation     since it leaves out of consideration functions that grow faster than power ones. According to our criterion, including Lagrangians with a fast growth  than power functions  is one of the greatest achievements of this paper.  For this reason, we present a proof  of the results obtained in \cite{chamra2017anisotropic} without the assumption $\Phi\in\Delta_2$. When $\Phi\notin\Delta_2$, the  differentiability of $I$ is somewhat more delicate since the effective domain of $I$ is not the whole space $\wphi$.

\begin{defi}\label{def:estructura} We say that a Lagrangian $\mathcal{L}$ satisfies the condition \eqref{eq:condicion-estructura} if
\begin{equation}\label{eq:condicion-estructura}
  |\mathcal{L}|+ |\nabla_{x}\mathcal{L}|+\Phi^{\star}\left(\frac{\nabla_{y}\mathcal{L}}{\lambda}\right)
\leq
a(x)\left[b(t)+ \Phi\left(\frac{y}{\Lambda}\right)\right], \tag{$S$}
\end{equation}
for a.e. $t\in [0,T]$,
where  $a\in C\left(\mathbb{R}^d,[0,+\infty)\right)$, $b\in L^1\left([0,T],[0,+\infty)\right) $ and $\Lambda,\lambda>0$.
\end{defi}

Condition \eqref{eq:condicion-estructura} includes structure conditions that have  been previously considered in the literature in the case of $p$-laplacian and $(p_1,p_2)$-laplacian systems. For example, it is easy to see that, when $\Phi(x)=\Phi_p(x)=|x|^p/p$, then  condition \eqref{eq:condicion-estructura}  is equivalent to the structure conditions in  \cite[Thm. 1.4]{mawhin2010critical}.  If $\Phi$ is a radial $N_{\infty}$ function such that $\Phi^{\star}$ satisfies  the $\Delta_2$-condition,  then \eqref{eq:condicion-estructura} is related to conditions  \cite[Eq. (2)-(4)]{ABGMS2015}.   If $\Phi=\Phi_{p_1,p_2}$ is as in Equation \eqref{eq:phip_1p_2} and $\mathcal{L}=\mathcal{L}(t,x_1,x_2,y_1,y_2)$ is a Lagrangian with $\mathcal{L}:[0,T]\times\rr^d\times\rr^d\times\rr^d\times\rr^d\to\rr$, then inequality \eqref{eq:condicion-estructura} is related to structure conditions like
\cite[Lemma 3.1, Eq. (3.1)]{Tian2007192}. As can be seen, condition \eqref{eq:condicion-estructura} is a more compact expression than \cite[Lemma 3.1, Eq. (3.1)]{Tian2007192} and moreover   weaker, because  \eqref{eq:condicion-estructura} does not imply a control of
$|D_{y_1}L|$ independent of $y_2$.

\begin{comentario} We leave to the reader the proof of the fact that if a Lagrange function $\mathcal{L}$ satisfies structure condition \eqref{eq:condicion-estructura} and $\Phi\strictif \Phi_0$, then $\mathcal{L}$ satisfies \eqref{eq:condicion-estructura} with $\Phi_0$ instead of $\Phi$ and possibly with other functions $b$, $a$ and constants $\Lambda$ and $\lambda$.
 \end{comentario}

\begin{comentario}\label{com:lphi-satis S}
 The Lagrangian $\mathcal{L}=\mathcal{L}_{\Phi,F}=\Phi(y)+F(t,x)$ satisfies condition  \eqref{eq:condicion-estructura}, for every $\Lambda<1$. In order to prove this, the only non trivial fact that we should establish is that $ \Phi^{\star}(\nabla_{y}\mathcal{L})
\leq
a(x)\left\{b(t)+ \Phi\left({y}/{\Lambda}\right)\right\}$. From  \ref{eq:young-id} and the fact that $(d/dt)\Phi(tx)=\nabla\Phi(tx)\cdot x$ is an non decreasing function of $t$, we obtain  

\[
  \begin{split}
     \Phi^{\star}(\nabla \Phi (x))\leq  x\cdot \nabla\Phi(x) \leq\frac{1}{\Lambda^{-1}-1}\int_1^{\Lambda^{-1}} \frac{d}{dt}\Phi(tx)dt \leq\frac{1}{\Lambda^{-1}-1}\Phi(\Lambda^{-1} x). \end{split}
\]
Therefore  $\Phi^{\star}(\nabla_{y}\mathcal{L})=\Phi^{\star}\left(\nabla\Phi(y)\right)\leq \Lambda(1-\Lambda)^{-1} \Phi(y/\Lambda )$, for every $\Lambda<1$. 
\end{comentario}

Given a function $a:\mathbb{R}^d\to \mathbb{R}$, we define the composition operator $\b{a}:\mathcal{M}\to \mathcal{M}$ by $\b{a}(u)(x)= a(u(x))$.
We will often use the following result whose proof can be performed as that of  Corollary 2.3 in \cite{ABGMS2015}.

\begin{lem}\label{lem:cota-a}
\label{a_bound} If $a\in C(\mathbb{R}^d,\mathbb{R}^+)$ then $\b{a}:\wphi\to L^{\infty}([0,T])$ is bounded.
More concretely,  there exists a non decreasing function $A:[0,+\infty)\to[0,+\infty)$ such that
 $\|\b{a}(u)\|_{L^{\infty}([0,T])}\leq A(\|u\|_{\wphi})$.
\end{lem}

The following lemma will be applied several times. We adapted the proof of 
\cite[Lemma 2.5]{ABGMS2015} to the anisotropic case.    For an alternative approach, we  suggest \cite{chamra2017anisotropic}.

\begin{lem}\label{segundo lema}
Let  $\{{u}_n\}_{n\in \mathbb{N}}$ be a sequence of  functions  converging to  ${u}\in \Pi(\ephi,\lambda)$  in the $\lphi$-norm. Then, there exist a subsequence
${u}_{n_k}$ and a real valued function $h\in L^1([0,T],\rr)$ such that ${u}_{n_k}\rightarrow {u} \quad\text{a.e.}$ and $\Phi({u}_{n_k}/\lambda)\leq h\quad\text{a.e.}$
\end{lem}

\begin{proof}
  Since $d({u},\ephi)<\lambda$ and ${u}_n$ converges to ${u}$, there exist   a subsequence of $u_n$ (again denoted $u_n$), $\overline{\lambda}\in (0,\lambda)$ and $u_0\in\ephi$ such that $d(u_n,u_0)<\overline{\lambda}$, $n=1,\ldots$. As $\lphi\left([0,T],\rr^d\right)\hookrightarrow L^1\left([0,T],\rr^d\right)$, the sequence $u_n$ converges in measure to $u$. Therefore, we can extract a subsequence (denoted again $u_n$) such that $u_n\to u$ a.e. and
  \[\lambda_{n}:=\|{u}_{n}-{u}_{n-1}\orlnor<\frac{\lambda-\overline{\lambda}}{2^{n-1}},\quad\hbox{for } n\geq 2.\]
 We can assume $\lambda_n>0$ for every $n=1,\ldots$. We write $\lambda_1:=\|u_1-u_0\orlnor$ and  $\lambda_0:=\lambda-\sum_{n=1}^{\infty}\lambda_n$, and  we define $h:[0,T]\rightarrow\mathbb{R}$  by
 \begin{equation}\label{eq:serie} h(t)=  \frac{\lambda_0}{\lambda}\Phi\left(\frac{u_0}{\lambda_0}\right)+\sum_{j=0}^{\infty}\frac{\lambda_{j+1}}{\lambda}\Phi\left(\frac{u_{j+1}-u_j}{\lambda_{j+1}}\right).
\end{equation}
 As $\Phi(0)=0$ and  $\Phi$ is a convex function, we have for any $n=1,\ldots$
\[
 \begin{split}
   \Phi\left(\frac{u_n}{\lambda}\right) &=\Phi\left(  \frac{u_0}{\lambda}+   \sum_{j=0}^{n-1}\frac{u_{j+1}-u_j}{\lambda} \right)\\
   &\leq
   \frac{\lambda_0}{\lambda}\Phi\left(\frac{u_0}{\lambda_0}\right)+\sum_{j=0}^{n-1}\frac{\lambda_{j+1}}{\lambda}\Phi\left(\frac{u_{j+1}-u_j}{\lambda_{j+1}}\right) \leq h.
 \end{split}
\]

Since $u_0\in\ephi\subset \claseor$ and $\ephi$ is a subspace, we get that $\Phi(u_0/\lambda_0)\in L^1([0,T],\rr)$.
On the other hand, $\|u_{j+1}-u_j\orlnor = \lambda_{j+1}$ and therefore
\[
\int_0^T\Phi\left(\frac{u_{j+1}-u_j}{\lambda_{j+1}}\right)dt\leq 1.
\]
Then $h\in L^1([0,T],\rr)$. \end{proof}

The  proof of the following theorem follows the same lines as \cite[Thm. 3.2]{ABGMS2015}, but with some modifications due to the lack of monotonicity of $\Phi$ with  respect to the euclidean norm and the fact that  the notion of absolutely continuous norm (used intensely in \cite[Thm. 3.2]{ABGMS2015}) does not work very well in the framework of anisotropic Orlicz spaces when $\Phi \notin \Delta_2$.

\begin{thm}\label{teo:diferenciabilidad}
Let $\mathcal{L}$ be a differentiable Carath\'eodory function satisfying \eqref{eq:condicion-estructura}.
Then the following statements hold:
\begin{enumerate}
\item \label{it:T1item1} \label{A1} The action integral given by \eqref{eq:integral_accion_gen}
is finitely defined on the set $\domi_{\Lambda}:=W^{1}\lphi\cap\{u|u'\in\Pi(\ephi,\Lambda)\}$.

\item\label{it:T1item3} The function  $I$ is G\^ateaux differentiable on $\domi_{\Lambda}$ and  its derivative $I'$ is demicontinuous from 
$\domi_{\Lambda}$  into $\left[\wphi \right]^{\star}$, i.e. $I'$ is continuous when $\domi_{\Lambda}$ is equipped with the strong topology and   $\left[\wphi \right]^{\star}$ with the weak${\star}$ topology. Moreover, $I'$ is given by the following expression
\begin{equation}\label{eq:DerAccion}
\langle  I'(u),v\rangle= \int_0^T \left[\nabla_x\mathcal{L}\big(t,u,u'\big)\cdot v
+ \nabla_y\mathcal{L}\big(t,u,u'\big)\cdot v'\right] \ dt.
\end{equation}

\item\label{it:T1item4}  If  $\Phi^{\star} \in \Delta_2$ then 
  $I'$ is continuous from $\domi_{\Lambda}$ into $\left[\wphi\right]^{\star}$ when both spaces are equipped with the strong topology.
\end{enumerate}
\end{thm}

\begin{proof}
Let $u\in \domi_{\Lambda}$. From \eqref{eq:inclusiones} we obtain that $\Phi( u'(t)/\Lambda) \in L^1$.
Now, from  \eqref{eq:condicion-estructura} and Lemma \ref{lem:cota-a}, we have
 \begin{equation}\label{eq:cota-condicion-estructura}
   \begin{split}
|\mathcal{L}(t,u(t),u'(t))|&+ |\nabla_x\mathcal{L}(t,u(t),u'(t))|
+\Phi^{\star}\left(\frac{\nabla_y\mathcal{L}(t,u,u')}{\lambda}\right)
\\
&\leq A(\|u\sobnor ) \left[b(t)+ \Phi\left(\frac{u'(t)}{\Lambda}\right) \right]\in
 L^1.
\end{split}
\end{equation}
Thus, by integrating this inequality item \eqref{it:T1item1} is proved.

We split up the proof of item \ref{it:T1item3} into four steps.

\noindent\emph{Step 1. The non linear operator  $u \mapsto \nabla_x\mathcal{L}(\cdot,u,u')$ is continuous from $\domi_{\Lambda}$ into $L^{1}([0,T])$ with the strong topology on both sets.}

Let   $\{u_n\}_{n\in \mathbb{N}}$ be a sequence of  functions in $\domi_{\Lambda}$  
and let $u\in \domi_{\Lambda}$  such that $u_n\rightarrow u$ in $\wphi$.
By \eqref{eq:sobolev},  $u_n \to u$ uniformly.
As $u'_n\rightarrow u'\in\domi_{\Lambda}$, by 
  Lemma \ref{segundo lema}, there exist a subsequence of  $u'_{n}$ (again denoted $u'_{n}$) and a function  
	$h\in L^1([0,T],\rr)$
	such that  $u'_{n}\rightarrow u' \quad\text{a.e.}$ and $\Phi(u'_{n}/\Lambda)\leq h\quad\text{a.e}$.  

Since $u_{n}$, $n=1,2,\ldots,$ is a bounded sequence in $\wphi$, according to  Lemma \ref{lem:cota-a}, 
there exists $M>0$ such that $\|\b{a}(u_{n})\|_{L^{\infty}} \leq M$, $n=1,2,\ldots$.  
From the previous facts and \eqref{eq:cota-condicion-estructura}, we get
\begin{equation*}\label{eq:DxL1-bis}
  |\nabla_x\mathcal{L}(\cdot,u_{n},u'_{n})|\leq a(u_{n})\left[b+\Phi\left(\frac{u'_{n}}{\Lambda}\right)\right]\leq
M (b+h) \in L^1.
\end{equation*}
On the other hand, by the continuous differentiability of $\mathcal{L}$, we have
\[\nabla_x\mathcal{L}(t,u_{n_k}(t),u'_{n_k}(t))\to \nabla_x\mathcal{L}(t,u(t),u'(t))\quad\hbox{ for a.e. } t\in[0,T].\]
Applying Lebesgue Dominated Convergence Theorem we conclude the proof of step 1.

\noindent\emph{Step 2. The non linear operator   $u
 \mapsto  \nabla_y\mathcal{L}(\cdot,u,u')$ is continuous from $\domi_{\Lambda}$ with the strong topology  
into $\left[\lphi\right]^{\star}$  with the weak$\star$ topology.}

 Let $u\in \domi_{\Lambda}$.  From  \eqref{eq:cota-condicion-estructura}, it follows that 
\begin{equation}\label{eq:DyLpsi-clase}
\nabla_y\mathcal{L}(\cdot,u,u')\in \lambda C^{\Phi^{\star}}\left([0,T],\rr^d\right)\subset\lpsi\left([0,T],\rr^d\right)\subset\left[\lphi\left([0,T],\rr^d\right)\right]^{\star}.
\end{equation}

Let $u_n,u\in \domi_{\Lambda}$ such that $u_n\to u$ in the norm of $\wphi$. 
We must prove that  $\nabla_y\mathcal{L}(\cdot,u_n,u'_n)\overset{w^\star}{\rightharpoonup} 
\nabla_y\mathcal{L}(\cdot,u,u')$. 
On the contrary, there exist $v\in\lphi$, $\epsilon>0$ and a subsequence of $\{u_n\}$ (denoted  $\{u_n\}$ for simplicity)  such that
\begin{equation}\label{cota_eps}
 \left| \langle \nabla_y\mathcal{L}(\cdot,u_n,u'_n),v \rangle - 
\langle  \nabla_y\mathcal{L}(\cdot,u,u'),v \rangle\right|\geq \epsilon.
\end{equation}
We have $u_n\rightarrow u$ in $\lphi$ and
$u'_n\rightarrow u'$ with $u'\in\Pi(\ephi,\Lambda)$.
 By Lemmas \ref{cor:unif_conv}  and \ref{segundo lema}, there exist a subsequence of $\{u_n\}$ (again denoted  $\{u_n\}$ for simplicity) 
and a function $h\in L^1([0,T],\rr)$ such that 
$u_n\rightarrow u$ uniformly, $u'_n\rightarrow u' \quad\text{a.e.}$ and $\Phi(u'_n/\Lambda)\leq h\quad\text{a.e.}$ 
As in the previous step,
Lemma \ref{lem:cota-a} implies that $a(u_n(t))$ is uniformly bounded by a certain constant $M>0$. 
Therefore,   from inequality  \eqref{eq:cota-condicion-estructura} with $u_n$ instead of $u$, we have 
\begin{equation}\label{eq:Dy-suc}
  \Phi^{\star}\left(\frac{\nabla_y\mathcal{L}(\cdot,u_n,u'_n)}{\lambda}\right)   
	\leq M (b+h)=:h_1\in L^1.
\end{equation}
As $v \in \lphi$ there exists $\lambda_v>0$ such that $\Phi(v/\lambda_v)\in L^1$. 
Now, by Young's inequality and \eqref{eq:Dy-suc}, we have
\begin{equation}\label{eq:Dy_lambda-Psi}
\begin{split}
\nabla_y\mathcal{L}(\cdot,u_{n},u'_{n})\cdot v(t)
&\leq 
\lambda\lambda_v\left[\Phi^{\star}\left(\frac{\nabla_y\mathcal{L}(\cdot,u_{n},u'_{n})}{\lambda}\right)+\Phi\left(\frac{v}{\lambda_v}\right)\right]
\\
&\leq \lambda\lambda_v M (b+h)+\lambda\lambda_v  \Phi\left(\frac{v}{\lambda_v}\right)\in L^1.
\end{split}
\end{equation}
  Finally, from  Lebesgue Dominated Convergence Theorem, we deduce
\begin{equation}\label{conv_debil}
\int_0^T  \nabla_y\mathcal{L}(t,u_{n},u'_{n})
\cdot  v dt
\to 
\int_0^T \nabla_y\mathcal{L}(t,u,u')\cdot v\, dt, \end{equation}
which contradicts the inequality \eqref{cota_eps}. This completes the proof of step 2.

\emph{Step 3.} We will prove \eqref{eq:DerAccion}. 
 Note that \eqref{eq:cota-condicion-estructura},  \eqref{eq:DyLpsi-clase} and the imbeddings $\wphi \hookrightarrow L^{\infty}$ and  
$\lpsi\hookrightarrow  \left[\lphi\right]^{\star}$ imply that the second member of
\eqref{eq:DerAccion} defines an element of $\left[\wphi\right]^{\star}$.

The proof follows similar lines as \cite[Thm. 1.4]{mawhin2010critical}. 
For $u\in \domi_{\Lambda}$ and $0\neq v\in\wphi$, we define the function
\[H(s,t):=\mathcal{L}(t,u(t)+s v(t),u'(t)+sv'(t)).\]

For  $|s|\leq s_0:=\left(\Lambda-d(u',\ephi)\right)/\|v\sobnor$  we have that  $u'+sv' \in \Pi(\ephi,\Lambda)$. 
This fact implies, in virtue of Theorem \ref{teo:diferenciabilidad} item \ref{it:T1item1}, 
that $I(u+s v)$ is well defined and finite for $|s|\leq s_0$.

We write $s_1:=\min\{s_0,1-d(u',\ephi)/\Lambda\}$. Let $\lambda_v>0$ such that $\Phi(v'/\lambda_v)\in L^1$. As $u'\in\Pi(\ephi,\Lambda)$, then

\[
d\left(\frac{u'}{(1-s_1)\Lambda},E^{\Phi}\right)=\frac{1}{(1-s_1)\Lambda}d(u', E^{\Phi})<1,
\]
and consequently  $(1-s_1)^{-1}\Lambda^{-1}u'\in C^\Phi$. Hence,  if $v'\in\lphi$ and $|s|\leq s_1 \Lambda\lambda_v^{-1}$, from the convexity of $\Phi$ and \ref{eq:escalar_ine_2}, we get
\begin{equation}\label{eq:cota-u+sv}
\begin{split}
\Phi\left(\frac{u'+sv'}{\Lambda}\right)&
\leq
(1-s_1)\Phi\left(\frac{u'}{(1-s_1)\Lambda}\right)+s_1 \Phi\left(\frac{s}{s_1\Lambda}v'\right)
\\
&\leq
(1-s_1)\Phi\left(\frac{u'}{(1-s_1)\Lambda}\right)+s_1 \Phi\left(\frac{v'}{\lambda_v}\right)\\
&=:h(t) \in L^1.
\end{split}
\end{equation}

We also have
$
\|u+sv\sobnor\leq \|u\sobnor+s_0\|v\sobnor;
$
then, by Lemma \ref{lem:cota-a}, there exists $M>0$ independent of $s$, such that
$\|a(u+sv)\linf\leq M$. Now, applying Young's inequality,  \eqref{eq:cota-condicion-estructura},
the fact that $v \in L^{\infty}$, \eqref{eq:cota-u+sv} and $\Phi(v'/\lambda_v)\in L_1$, 
we get
\begin{equation}\label{ctg}
\begin{split}
|D_s H(s,t)|&=\left| \nabla_x\mathcal{L}(t,u+sv,u'+sv')\cdot v +
  \nabla_y\mathcal{L}(t, u+s v, u'+sv')\cdot v'\right| \\
  & \leq M \left[ b(t)+ \Phi\left(\frac{u'+sv'}{\Lambda}\right)\right]|v|\\
 &\quad+ \lambda\lambda_v\left[\Phi^{\star}\left(\frac{\nabla_y\mathcal{L}(t,u+sv,u'+sv')}{\lambda}\right)+\Phi\left(\frac{v'}{\lambda_v}\right) \right]
\\
 &\leq 
 M \left[ b(t)+ \Phi\left(\frac{u'+sv'}{\Lambda}\right)\right] \left(|v|+\lambda\lambda_v\right) +\lambda\lambda_v \Phi\left(\frac{v'}{\lambda_v}\right)\\
  &\leq 
 M \left( b(t)+h(t)\right) \left(|v|+\lambda\lambda_v\right)+\lambda\lambda_v \Phi\left(\frac{v'}{\lambda_v}\right)
 \in L^1.
\end{split}
\end{equation}

Consequently, $I$ has a directional derivative and
\[
\langle I'(u),v \rangle=\frac{d}{ds}I(u+s v)\big|_{s=0}=\int_0^T  
\left[\nabla_x\mathcal{L}(t,u,u')\cdot v+ \nabla_y\mathcal{L}(t,u,u')\cdot v'\right] \ dt.
\]
Moreover, from the previous formula, \eqref{eq:cota-condicion-estructura},  \eqref{eq:DyLpsi-clase} and
Lemma \ref{lem:inclusion orlicz}, we obtain
\[
|\langle I'(u),v \rangle| \leq \|\nabla_x\mathcal{L}\|_{L^1} \| v\linf + 
\|\nabla_y\mathcal{L}\|_{\lpsi} \|v'\orlnor \leq C \|v\sobnor,
\]
with an appropriate constant $C$.
This completes the proof of the G\^ateaux differentiability of $I$. 
The previous steps imply the demicontinuity of the operator $I':\domi_{\Lambda}  \to \left[\wphi_d
\right]^{\star} $.

In order to prove item  \ref{it:T1item4}, it is necessary to see that the maps \linebreak[4]$u\mapsto \nabla_x\mathcal{L}(t,u,u')$  
and $u\mapsto \nabla_y\mathcal{L}(t,u,u')$  are norm continuous
from $\domi_{\Lambda} $ into $L^1$ and
 $\lpsi$, respectively.  
It remains to prove the continuity of the second map. 
To this purpose, we take  $u_n, u \in \domi_{\Lambda}$, $n=1,2,\dots$, with $\|u_n- u\sobnor\to 0$.  
As before, we can deduce the existence of a subsequence (denoted $u'_n$ for simplicity) and $h_1 \in L^1$ such that \eqref{eq:Dy-suc} holds and $u_n \to u$ a.e.
 Since $\Phi^{\star}\in\Delta_2$, we have 
\begin{equation}
\Phi^{\star}(\nabla_y \mathcal{L}(\cdot,u_n,u'_n))\leq c(\lambda) \Phi^{\star}\left(\frac{\nabla_y \mathcal{L}(\cdot,u_n,u'_n)}{\lambda}\right)+1\leq c(\lambda)h_1+1=:h_2\in L^1.
\end{equation} 
Then, from \ref{eq:quasi-sub-aditividad}, we get
\[\Phi^{\star}\left(\nabla_y \mathcal{L}(\cdot,u_n,u'_n)-
\nabla_y \mathcal{L}(\cdot,u,u')\right)\leq K (h_2+\Phi^{\star}(\nabla_y \mathcal{L}(\cdot,u,u')))+1.\]
Now, by Lebesgue Dominated Convergence Theorem, we obtain 
$\nabla_y \mathcal{L}(\cdot,u_n,u'_n)$ is $\rho_{\Phi^{\star}}$ modular convergent to $\nabla_y \mathcal{L}\left(\cdot,u,u')\right)$, i.e.
$\rho_{\Phi^{\star}}(u_n-u)\to 0$. 
Since $\Phi^{\star} \in \Delta_2$, modular convergence implies norm convergence (see \cite{Skaff1969}).
\end{proof}

\begin{proof}[\textbf{Proof Theorem \ref{th:tor_prin}}] Suppose that $d(u',L^{\infty})<1$.
Since $d(u',\ephi)=d(u',L^{\infty})$, according to Remark \ref{com:lphi-satis S} and Theorem \ref{teo:diferenciabilidad}, $I$ is G\^ateaux differentiable at $u$. By Fermat's rule (see \cite[Prop. 4.12]{clarke2013functional}), we have $\langle I'(u),v\rangle=0$ for every $v\in H$. Therefore

\begin{equation}\label{eq:der_deb}\int_0^T\nabla\Phi(u'(t))\cdot v'(t)dt=-\int_0^T \nabla_xF(t,u(t))\cdot v(t)dt.\end{equation}

From Theorem \ref{teo:diferenciabilidad}, $\nabla_xF(t,u(t))\in L^1([0,T],\rr^d)$ and  $\nabla\Phi\left(u'(t)\right)\in\linebreak\lpsi([0,T],\rr)\hookrightarrow L^1([0,T],\rr)$.  Identity \eqref{eq:der_deb} holds for every $v\in C^{\infty} ([0,T],\rr^d)$ with $v(0)=v(T)$. Using 
\cite[Fundamental Lemma, p. 6]{mawhin2010critical}, we get that $\nabla\Phi(u'(t))$ is absolutely continuous and $(d/dt)\left(\nabla\Phi(u'(t))\right) = \nabla_xF(t,u(t))$ a.e. on $[0,T]$. Moreover, $\nabla\Phi(u'(0))=\nabla\Phi(u'(T))$. Since  $\Phi$ is \emph{strictly convex},  then  $\nabla\Phi:\mathbb{R}^d\to\mathbb{R}^d$ is a one-to-one map  (see, e.g. \cite[Ex. 4.17, p. 67]{clarke2013functional}). Hence, we conclude that $u'(0)=u'(T)$. Finally, Theorem \ref{th:tor_prin} is proven.
\end{proof}

\begin{comentario} The condition $d(u',L^{\infty})< 1$ is trivially satisfied when $\Phi \in \Delta_2$ because, in this case, $L^{\infty}$ is dense in $\lphi([0,T],\rr^d)$. In particular, our Theorem \ref{th:tor_prin} implies existence of solutions, among others, for the $(p_1,p_2)$-laplacian system. 
\end{comentario}

It is possible to use the regularity theory in order to get  that minimizers $u$ of $I$ satisfy $u'\in L^{\infty}$. To cite one example, we have the following result.

\begin{cor} Let $\Phi,F$ and $H$ be as in Theorem \ref{th:tor_prin}. Additionally, suppose that
$F(t,x)$ is differentiable with respect to $(t,x)$ and that 
\begin{equation}\label{eq:apriorilip}\left| \frac{\partial}{\partial t} F(t,x)\right|  \leq a(x)b(t),\end{equation}
with $a$ and $b$ as in \ref{item:condicion_a}.
 If $u$ is a minimum of $I$ on the set $H$ then $u$ is solution of \eqref{eq:ProbPhiLapla}. 
 
\end{cor}

\begin{proof} We note that $u$ is a minimum of $I$ on the set defined by a Dirichlet boundary condition
\[\{v\in\wphi([0,T],\rr^d)| v(0)=u(0), v(T)=u(T)\}.\]
Therefore, we can apply Proposition 3.1 in \cite{clarke1985regularity} (see also the following remark) and we obtain $u'\in L^{\infty}$.
\end{proof}

\begin{comentario}\label{com:ejemplo} Returning to the system \eqref{eq:RandExample} of Example \ref{ex:RandEx},
we note that the $N_{\infty}$ function $\Phi(y_1,y_2)=\exp(y_1^2+y_2^2)-1$ has a  complementary function which satisfies the $\Delta_2$-condition (see \cite[p. 28]{KR}).  In addition, for every $p>1$ we have $|(y_1,y_2)|^p \llcurly \Phi(y_1,y_2)$. Therefore $\Phie(y_1,y_2)\llcurly|(y_1,y_2)|^q $ for $q=p/(p-1)$. Consequently,  if $F(t,x_1,x_2)=P(t)Q(x_1,x_2)$ with $P$ and $Q$ polynomials, and $d(t):=C\max\{1,|P(t)|\}$  then $\Phie(d^{-1}(t)\nabla_xF)\leq |(x_1,x_2)|^q+1$,  where $p$ and $C$ are chosen large enough. Hence $\Phi$ and $F$ satisfy \eqref{eq:cond-sub} with $\Phi_0(y_1,y_2)=|(y_1,y_2)|^p$.  The conditions  \ref{item:condicion_c}, \ref{item:condicion_a}  and \eqref{eq:apriorilip} can be proved in a direct way. All these facts show that the assessment of Example \ref{ex:RandEx} is true.
 
\end{comentario}

\section*{Acknowledgments}
The authors are partially supported by  UNRC and UNLPam grants. The second author is  partially supported by a  UNSL grant.

 \bibliographystyle{plain}
 
\bibliography{biblio}

\end{document}